\documentclass[12pt]{amsart}
\usepackage{amssymb}
\usepackage{amsmath}

\usepackage{color}

\newtheorem{theorem}{Theorem}[section]
\newtheorem{lemma}[theorem]{Lemma}
\newtheorem{proposition}[theorem]{Proposition}
\newtheorem{corollary}[theorem]{Corollary}
\theoremstyle{definition}
\newtheorem{definition}[theorem]{Definition}
\newtheorem{example}[theorem]{Example}

\theoremstyle{remark}
\newtheorem{remark}[theorem]{Remark}
\numberwithin{equation}{section}

\def\id{{\bf 1}\!\!{\rm I}}

\def\bc{{\mathbb C}}

\def\bn{{\mathbb N}}

\def\br{{\mathbb R}}

\def\l{\lambda}

\def\i{\varepsilon}
\def\t{\tau}
\def\f{\varphi}

\def\a{\alpha}

\def\m{\mu}

\def\d{\delta}

\def\ck{\mathcal{K}}

\begin{document}
\setcounter{page}{1}

\title[Rates of convergence]{Spectral conditions for uniform $P$-ergodicities of Markov operators on abstract states spaces}

\author[Nazife Erkur\c{s}un-\"Ozcan, Farrukh Mukhamedov]{Nazife Erkur\c{s}un-\"Ozcan$^1$ and Farrukh Mukhamedov$^2$$^*$}

\address{$^{1}$ Department of Mathematics, Faculty of Science, Hacettepe University, Ankara, 06800, Turkey}
\email{{erkursun.ozcan@hacettepe.edu.tr}}

\address{$^{2}$ Department of Mathematical Sciences, College of Science, United Arab Emirates University 15551, Al Ain, Abu Dhabi, United
Arab Emirates} \email{{far75m@yandex.ru; farrukh.m@uaeu.ac.ae}}




\date{Received: xxxxxx; Revised: yyyyyy; Accepted: zzzzzz.
\newline \indent $^{*}$ Corresponding author}

\begin{abstract}
In the present paper deals with asymptotical stability of Markov operators acting on abstract state spaces (i.e. an ordered Banach space, where the norm has an additivity
property on the cone of positive elements). Basically, we are interested in the rate of convergence when a Markov operator $T$ satisfies the uniform $P$-ergodicity, i.e. $\|T^n-P\|\to 0$, here $P$ is a projection. We have showed that $T$ is uniformly $P$-ergodic if and only if $\|T^n-P\|\leq C\beta^n$, $0<\beta<1$. In this paper, we prove that such a $\beta$ is characterized by the spectral radius of
$T-P$. Moreover, we give Deoblin's kind of conditions for the
uniform $P$-ergodicity of Markov operators.
\end{abstract} \maketitle

\section{Introduction}

The present work is a continuation of the paper \cite{MA} where we
have introduced a generalized Dobrushin ergodicity coefficient
$\d_P(T)$ of Markov operators (acting on abstract state spaces)
with respect to a projection $P$, and studied its properties. In \cite{MA} we have characterized the uniform $P$-ergodicity of a Markov operator, i.e.
$\|T^n-P\|\to 0$ in terms of $\d_P(T)$. If $P$ is a rank one projection, then such kind of ergodicity has been intensively studied
by many authors \cite{B,E,Mit,Mit1,Paz}.  When $P$ is not a rank one projection, then it
turned out that the introduced coefficient was more effective than
the usual Dobrushin's ergodicity coefficient (see \cite{D}). We stress that investigations on the asymptotic stability of Markov operators to projections were only considered when projections were taken compact ones.  However, in the general setting, there were a few papers (see, for example \cite{H}).
Therefore, our investigation is more general and allows to gets interesting results in both classical and non-commutative settings.

On the other hand, as soon as we have the uniform $P$-ergodicity, it is natural to study the rate of convergence of the quantity $\beta$ in
$\|T^n-P\|\leq C\beta^n$. We point
out if $T$ represents a discrete Markov chain and $P$ is a rank one projection, then the best possible value for $\beta$ is characterized by the spectral radius of
$T-P$ \cite{IL,Num}.  Main aim of this paper is to establish a similar kind of estimation for $\beta$ in a general setting. Namely, we consider a much more general situation, where Markov operator $T$ acts on some abstract state space and $P$ is also some Markov projection acting on the same space.
In what follows, by an abstract state space it is
meant an ordered Banach space, where the norm has an additivity
property on the cone of positive elements. Examples of these
spaces include all classical $L^1$-spaces and the space of density
operators acting on some Hilbert spaces \cite{Alf,Jam}. Moreover,
any Banach space can be embedded into some abstract spaces (see
Appendix, Example \ref{E13}).
We notice that the consideration of these types of Banach spaces is convenient and important for the study of several properties of physical and
probabilistic processes in an abstract framework which covers the classical and quantum cases (see \cite{Alf,E}). In this
setting, limiting behaviors of Markov operators were
investigated in \cite{B,EW1,EM2017,EM2018,R,SZ}.

The paper is organized as follows. In Section 2, we provide preliminary definitions and results on properties of the generalized Dobrushin coefficient of a Markov operator acting on abstract state spaces. Section 3 we establish that $\beta$ is characterized by the spectral radius of
$T-P$.  We point out that the obtained
results are even new in the classical (when the projection is not rank one) and quantum
settings (comp. \cite{RKW}). Furthermore, in Section 4, we give other kind (more constructive) of necessary and sufficient conditions for the
uniform $P$-ergodicity of Markov operators.

\section{Preliminaries}

In this section, a few necessary definitions and facts about the ordered Banach spaces are collected.

Let $X$ be an ordered vector space over $\br$
with a cone $X_+=\{x\in X: \ x\geq 0\}$. A subset $\ck$ is called
a {\it base} for $X$, if one has $\ck=\{x\in X_+:\ f(x)=1\}$ for
some strictly positive (i.e. $f(x)>0$ for $x>0$) linear functional
$f$ on $X$. An ordered vector space $X$ with generating cone $X_+$
(i.e. $X=X_+-X_+$) and a fixed base $\ck$, defined by a functional
$f$, is called {\it an ordered vector space with a base}
\cite{Alf}. Hereinafter, we denote it as  $(X,X_+,\ck,f)$. Let
$U$ be the convex hull of the set $\ck\cup(-\ck)$, and let
$$
\|x\|_{\ck}=\inf\{\l\in\br_+:\ x\in\l U\}.
$$
Then one can see that $\|\cdot\|_{\ck}$ is a seminorm on $X$.
Moreover, one has $\ck=\{x\in X_+: \ \|x\|_{\ck}=1\}$,
$f(x)=\|x\|_{\ck}$ for $x\in X_+$.
Assume that the seminorm becomes a norm, and $X$ is complete with respect to this norm, and $X_+$ is closed. Then
$(X,X_+,\ck,f)$ is called \textit{an abstract state
space}. In this case, $\ck$ is a closed face of the unit ball of $X$,  and $U$ contains the open
unit ball of $X$.
If the set $U$ is \textit{radially compact} \cite{Alf}, i.e.
$\ell\cap U$ is a closed and bounded segment for every line $\ell$
through the origin of $X$,  then  $\|\cdot\|_{\ck}$ is a
norm. The radial compactness is equivalent to the coincidence of $U$ with the closed unit ball of $X$.  In this case,
$X$  is called a \textit{strong abstract state
space}.  In the sequel, for the sake of simplicity, instead of
$\|\cdot\|_{\ck}$, the standard notation $\|\cdot\|$ is used.
To better understand the difference between a strong abstract state space and a more general class of base norm spaces, the reader is referred to
\cite{Yo,WN}.

Let $(X,X_+,\ck,f)$ be an abstract state space. A linear operator $T:X\to X$ is
called \textit{positive}, if $Tx\geq 0$ whenever $x\geq 0$. A positive linear
operator $T:X\to X$ is said to be a {\it Markov operator}, if $T(\ck)\subset\ck$.
From this, it is clear that $\|T\|=1$, and its adjoint mapping $T^*: X^*\to
X^*$ acts in an ordered Banach space $X^*$ with unit $f$.

Let $(X,X_+,\ck,f)$ be an abstract state space and let $T:X\to X$
be a Markov operator. Consider a projection
operator $P: X\to X$ (i.e. $P^2=P$). According to \cite{MA}
$T$  is called \textit{uniformly $P$-ergodic} if
$$
\lim_{n\to\infty}\|T^{n}-P\|=0.
$$
From this definition we immediately find that $P$ must be a Markov
projection.

We note that if $P=T_y$, for some $y\in X_+$, where
$T_y(x)=f(x)y$, then the uniform $P$-ergodicity coincides with
uniform ergodicity or uniform asymptotical stability considered in
\cite{M0,M01}.

Let $(X,X_+,\ck,f)$ be an abstract state space, and $T:X\to X$ be a Markov operator. Then the Dobrushin's coeefficient of $T$ is given by
\begin{equation}
\label{db} \d(T)=\sup_{x\in N,\ x\neq 0}\frac{\|Tx\|}{\|x\|}
\end{equation}
where \begin{equation}
\label{NN} N=\{x\in X: \ f(x)=0\}.
\end{equation}
 It is noticed that $\d(T)$
has been introduced and investigated in \cite{M0,M01}.

In \cite{MA}, it has been given a generalized version of the Dobrushin's ergodicity coefficient.

Let $(X,X_+,\ck,f)$ be an abstract state space and let $T:X\to X$
be a linear bounded operator and $P$ be a non-trivial projection
operator on $X$. Then
\begin{equation} \label{Dbp} \d_P(T)=\sup_{x\in N_P,\ x\neq
0}\frac{\|Tx\|}{\|x\|},
\end{equation}
where
\begin{equation}\label{Np} N_P=\{x\in X: \
Px=0\}.
\end{equation}

If $P=I$, we put $ \d_P(T)=1$. The quantity $\d_P(T)$ is called
the \textit{generalized Dobrushin ergodicity coefficient of $T$
with respect to $P$}.

We notice that if $X=\br^n$, then there are some formulas to
calculate this coefficient (see \cite{H,HR}).

\begin{remark}
Let $y_0\in \ck$ and consider the projection $Px=f(x)y_0$. Then
one can see that $N_P$ coincides with
\[N= \{ x\in X;\ f(x)=0\},\]
and in this case $\d_P(T)=\d(T)$. Hence, $\d_P(T)$ indeed is a
generalization of $\d(T)$.
\end{remark}

\begin{remark}
Let $P$ be a Markov projection on $X$. Then, for any Markov
operator $T:X\to X$
\[\d_P(T) \leq \d(T).\]
\end{remark}

Using $\d_P$, we define weak $P$-ergodicity of $T$. Namely, a Markov operator $T:X\to X$ is called {\it weakly $P$-ergodic}
if
$$
\lim_{n\to\infty}\d_P(T^{n})=0.
$$

We point out that the relations between uniform and week
$P$-ergodicities are discussed in \cite[Section 6]{MA}.

By $\Sigma (X)$ we denote the set of all Markov operators on $X$, and $\Sigma_P (X)$ the set of all Markov operators $T$ on $X$ with $PT =TP$.

The next result establishes several properties of the Dobrushin
ergodicity coefficient.

\begin{theorem}\label{Dob}
Let $(X, X_+ , \ck ,f)$ be an abstract state space, $P$ be a projection on $X$ and let  $T,S  \in \Sigma (X)$. The following statements hold:
\begin{itemize}
\item[(i)] $0 \leq \d_P (T) \leq 1$;
\item[(ii)] $|\d_P (T) - \d_P (S)| \leq \d_P (T-S) \leq \left\| T - S\right\|$;
\item[(iii)] if $H: X\to X$ is a linear bounded
operator such that $HP = PH$, then
$$
\d_P (TH) \leq \d_P (T) \left\| H \right\| ;
$$
\item[(iv)] if $H: X\to X$ is a linear bounded
operator such that $PH =0$, then
$$
\left\| TH\right\| \leq \d_P (T) \left\| H \right\| ;
$$
\item[(v)] if $S \in \Sigma_P (X)$, then
$$
\d_P (TS) \leq \d_P (T) \d_P (S).
$$
\end{itemize}
\end{theorem}

\begin{remark}\label{SS}
Let $X$ be a strong abstract state space (i.e. $X_+$ is 1-generating) and let $P$ be a Markov projection. Then
\begin{equation}\label{SS1}
\d_P (T) = \frac{1}{2} \sup \{\left\| Tu -
Tv\right\| : u,v \in \ck, \  \  u-v \in N_P\} .
\end{equation}
\end{remark}

Let $X$ be an abstract state space. Its complexification $\tilde X$ is defined by $\tilde{X}=X+iX$ with a reasonable norm $\|\cdot\|_\bc$ (see \cite{MST} for details). In this setting,
$X$ is called the \emph{real part} of
$\tilde X$. The \emph{positive cone} of $\tilde X$ is defined as $X_+$.  A vector $f \in \tilde X$ is called \emph{positive}, which
we denote by $f \geq 0$, if $f \in X_+$.  For two elements $f,g\in \tilde X$ we write, as usual, $f \leq g$ if $g-f \geq 0$. In the dual space $\tilde X^*$ of $\tilde X$, one can
introduce an order as follows: a functional $\varphi \in \tilde X^*$ fulfils $\varphi \geq 0$ if and only
if $\langle \varphi, x\rangle \geq 0$ for all $x \in X_+$; we denote the
positive cone in $\tilde X^*$ by $\tilde X^*_+:= (\tilde X^*)_+$.
In what follows, we assume that the norm $\|\cdot\|_\bc$ is taken as
\[\|x+iy\|_\infty = \sup_{0\leq t\leq 2\pi} \|x\cos t-y\sin t\|.\]
We note that all other complexification norms on $\tilde{X}$ are equivalent to $\|\cdot\|_\infty$, and moreover, $\|\cdot\|_\infty$ is the smallest one among all reasonable norms.

A linear mapping $T:X\to X$ can be uniquely extended to $\tilde{T}:\tilde{X} \to \tilde{X}$ by $\tilde{T}(x+iy)=Tx+iTy$. The operator $\tilde{T}$ is called the \textit{extension}
of $T$ and it is well-known that $\|\tilde{T}\|=\|T\|$. In what follows, a mapping $\tilde{T}:\tilde{X}\to \tilde{X}$ is called \textit{Markov}
if it is the extension of a Markov operator $T$. Let $\tilde{P}$ be the extension of a projection $P:X\to X$, and define
\[\tilde{\d }_{\tilde{P}}(\tilde{T})=\sup_{x\in N_{\tilde{P}}}\frac{\|\tilde{T}x\|_\infty}{\|x\|_\infty},\]
where $N_{\tilde{P}}=\{x\in \tilde{X};\ \tilde{P}x=0 \}$.

\begin{lemma}\cite{MA}
Let $X$ be a normed space, $T:X\to X$ be an operator and let $\tilde{T}$ be its extension. Then
\[\tilde{\d }_{\tilde{P}}(\tilde{T})= \d_P(T).\]
\end{lemma}

Now, let $S\in \Sigma (X)$ and let $P$ be a projection on $X$. Recall that $X=PX\oplus(I-P)X$ and so the dual $X^*=(PX)^*\oplus((I-P)X)^*$. Assume  that $\l$ is an eigenvalue of $S$, in the following we discuss the comparison between $|\l|$ and $\d_P(T)$.

\begin{theorem}\cite{MA}
Let $P$ be a  Markov projection on a complex space $X$ and let $S \in \Sigma_P(X)$. If one of the following conditions is satisfied:
\begin{enumerate}
\item[(i)]  $\l \neq 1$ is an eigenvalue of $S$ in $(I-\tilde{P})\tilde{X}$; or
\item[(ii)]  $\l \neq 1$ is an eigenvalue of $S^*$ in $((I-\tilde{P})\tilde{X})^*$,
\end{enumerate}
then $|\l|\leq \d_P(S)$.
\end{theorem}

\begin{remark}
We notice that spectral radius of uniformly $P$-ergodic operators will be considered in the next section.
\end{remark}

\begin{remark} We stress that there are many works devoted to the spectral properties of Markov operators (see for example, \cite{A,G}). One of them is its spectral gap. Namely, we say that a Markov operator $T$  on $X$ (here $X$ is a complex abstract state space) has
a \textit{spectral gap}, if one has $\|T(I-P)\|<1$, where $P$ is a Markov projection such that $PT=TP=P$. This is clearly equivalent to $\d_P(T)<1$.  When $X$ is taken as a non-commutative $L_p$-spaces, the spectral gap of Markov operator has been
recently studied in \cite{CPR}. In the classical setting, this gap has been extensively investigated by many authors (see for example, \cite{KM}).

We can stress that if $T$ has a spectral gap, then 1 has to be an isolated point of the spectrum. Indeed, choose an arbitrary $\i>0$ with $\i<1-\d_P(T)$. Assume that $\l$ is an element of the spectrum of $T$ such that $|1-\l|<\i$ with
corresponding eigenvector $x$. Then, it is clear that $y=x-Px$ belongs to $N_P$, therefore, one gets
\begin{eqnarray*}
Ty=Tx-TPx=Tx-PTx=\l(x-Px)=\l y
\end{eqnarray*}
hence, $y$ is an eigenvector with eigenvalue of $\l$, and we have
$$
\|Ty\|=|\l|\|y\|>\d_p(T)\|y\|,
$$
which contradicts to $\d_P(T)<1$.

We just emphasize that if $T$ has a spectral gap, then one has $\|T^n-P\|\to 0$, i.e. $T$ is uniform $P$-ergodic. Next section will be devoted to this notion.
\end{remark}

\begin{definition}
Let $T$ be a bounded operator on $X$. Then one defines
\begin{itemize}
\item[(i)] spectrum of  $T$:
$$\sigma (T) = \{\lambda\in\mathbb C : \lambda I - T \ \text{does not have a continuous inverse}\};$$
\item[(ii)] point spectrum of $T$: $$\sigma_p (T) = \{\lambda\in\mathbb C : \lambda x = Tx  \  \text{for} \  x\neq 0\};$$
\item[(iii)] the approximate spectrum of $T$:  $$\sigma_{app} (T) = \{\lambda\in\mathbb C : \left\| (\lambda I - T) x_n \right\| \to 0 \ \text{for some sequence} \  \{x_n\}\  \text{with} \  \left\| x_n \right\| = 1\}$$
\end{itemize}
\end{definition}

\textit{The spectral radius} of $T$ is defined to be $r(T)= \sup_{\lambda\in\sigma (T)} |\lambda |$ .
Let $T$ be a Markov operator, then by \textit{the rate of convergence},  we mean
\begin{equation}\label{beta}
\beta^* = \sup \{|\lambda |: \lambda \in \sigma (T), \lambda \neq 1\}.
\end{equation}


\section{Spectral condition for uniform $P$-ergodicity}

In this section, we are going to establish spectral conditions for the uniform $P$-ergodicity of Markov operators.  Basically, we study the best possible rate of convergence
$\|T^n-P\|$ while $T$ is uniform $P$-ergodic.  In what follows, we always assume that abstract state $X$ is considered over the complex field.

\begin{lemma} \label{lemmaspectrum1}
Let $(X, X_+, \ck, f)$ be an abstract state space and let $P$ be a Markov projection on $X$.
If $T \in \Sigma_P(X)$ with $TP =P$ and $\d_P (T^{n_0}) < 1$ for some  $n_0 \in \mathbb N$ then $r(T-P) < 1$.
\end{lemma}

\begin{proof}
From Theorem \ref{Dob}(i),(ii) we find
\begin{eqnarray}\label{pp}
&&\|T^n(I-P)\| \leq \d_P (T^n)\|I - P\|  \leq 2 \d_P (T^n) \  \  \  \text{and} \  \  \  \d_P (T^n) \leq  \|T^n - P\|.
\end{eqnarray}
By \cite[Proposition 2.11]{MA}, if $\d_P (T^{n_0} ) < 1$ then $\lim_{n\to\infty} \|T^n (I - P)\| =0$. Hence, using \eqref{pp} and $PT=TP=P$ one gets
\begin{eqnarray*}
r(T-P) &=& \lim_{n\to\infty} \|(T-P)^n\|^{1/n} \\
&=& \lim_{n\to\infty} \|T^n-P \|^{1/n} \\
&=& \lim_{n\to\infty} \|T^n- T^nP \|^{1/n} \\
&\leq& \lim_{n\to\infty} ( 2 \d_P (T^n))^{1/n} \\
&\leq& \limsup_{n\to\infty} ( 2 \d_P (T^n))^{1/n} \\
&\leq& \lim_{n\to\infty} ( 2 \|T^n - P\|)^{1/n} = r(T-P).
\end{eqnarray*}
This implies
$$
r(T-P)= \lim_{n\to\infty} (\d_P (T^n))^{1/n}.
$$
Consequently, we infer that $\d_P (T^{n_0}) < 1$,  for some $n_0 \in\mathbb N$,
if and only if $$r(T-P) <1.$$ This completes the proof.
\end{proof}

\begin{remark} We point out that, in reality,  due to $TP=PT=P$, to calculate the spectrum of $T-P$ it is enough to consider $T-P$ over $(I-P)(X)$.
\end{remark}

\begin{proposition}\label{theoremspectralradius1}
Let $T \in \Sigma (X)$ and $T$ be uniformly-P-ergodic. Then $r(T-P) < 1$.
\end{proposition}

\begin{proof}
Since $T$ is uniformly $P$-ergodic, then $TP=PT=P$ and there exists $n_0 \in\mathbb N$ such that $\d_P (T^{n_0}) <1$ . So Lemma \ref{lemmaspectrum1} implies the assertion.
\end{proof}

\begin{theorem}\label{rT}
Let $T\in \Sigma_P (X)$. Then the following conditions are equivalent.
\begin{itemize}
\item[(i)] $T$ is uniformly-P-ergodic;
\item[(ii)] $TP=P$ and there exists $n_0 \in \mathbb N$ such that $\d_P (T^{n_0}) < 1$;
\item[(iii)] $TP=P$ and $r (T-P) < 1$.
\end{itemize}
\end{theorem}

\begin{proof} The implications (i)$\Leftrightarrow$ (ii) follows from \cite[ Corollary 4.7]{MA}. The implication (ii)$\Rightarrow $(iii) immediately follows
from Lemma \ref{lemmaspectrum1}.
Let us establish (iii)$\Rightarrow$ (ii). Due to
$$\lim_{n\to\infty} (\d_P (T^n) )^{1/n} = r(T-P) = \inf_n (\d_P (T^n))^{1/n}
$$ and $r(T-P) < 1$, then there exists $n_0 \in\mathbb N$ such that $\d_P (T^{n_0}) <1$.
 \end{proof}

By  \cite[ Proposition 4.10]{MA}, we obtain the following result.

\begin{corollary}
Let $T \in\Sigma_P (X)$ with $TP=P$. Then $T$ is weakly-P-ergodic if and only if $r(T-P) <1$.
\end{corollary}

\begin{lemma}\label{approximatespectrum}
Let $P$ be a Markov projection and $T\in \Sigma_P (X)$ with $TP=P$. Then for $\lambda\neq 0$ and $\lambda\neq 1$, $\lambda \in \sigma_{app} ( T-P)$
if and only if $\lambda \in \sigma_{app} (T)$.
\end{lemma}

\begin{proof}
Assume that $\lambda\in\sigma_{app} (T-P) $. Then, there exists $x_n$ with $\|x_n\|=1$ such that
$\|(T-P-\lambda I) x_n\| \to 0$ as $n\to\infty$. This implies $\|P (T-P-\lambda I) x_n\| \to 0$ as $n\to\infty$. Due to
$TP=PT=P$, one gets $\|\lambda P x_n\| \to 0$. Since $\lambda\neq 0$, we have $\|P x_n\|\to 0$ as  $n\to\infty$. From
$$\|(T-\lambda I) x_n\|\leq \|(T-P-\lambda I) x_n\|+\|P x_n\|$$ it follows that $\|(T-\lambda I) x_n\| \to 0$, i.e. $\lambda \in\sigma_{app} (T)$.

Conversely, let us assume that $\lambda\in\sigma_{app} (T)$, which means $\|(T- \lambda I) x_n\| \to 0$ as $n\to\infty$. Therefore,
$$
\|P ( T-\lambda I) x_n\| = \|(PT- \lambda P) x_n\| = \|( P - \lambda P) x_n\| = \|(1-\lambda) P x_n\| \to 0$$
as $n\to\infty$. Due to  $\lambda\neq 1$, we obtain $\|P x_n\| \to 0$. Hence
$$
\|( T-P- \lambda I) x_n\| \leq \|( T- \lambda I) x_n\|+ \|P x_n\|\to 0 \ \ \textrm{as}\ \ n\to\infty.
$$
This means $\lambda \in\sigma_{app} (T-P)$.
\end{proof}

\begin{theorem}\label{rateofspectrumtheorem}
Let $T \in \Sigma (X)$ and let $P$ be a projection. If $T$ is uniformly $P$-ergodic, then $\beta^* = r(T-P)$, where $\beta^*$ is given by \eqref{beta}.
\end{theorem}

\begin{proof} According to Theorem \ref{rT}, we infer $r(T-P)<1$. So, $1\notin \sigma_{app} (T-P)$.
For the bounded operator, the boundary of spectrum is a subset of approximate point spectrum. Hence, by by Lemma \ref{approximatespectrum}, we obtain
\begin{eqnarray*}
r(T-P) &=& \sup \{|\lambda| : \lambda \in \sigma_{app} (T-P)\} \\
&=& \sup \{|\lambda| : \lambda \in \sigma_{app} (T) ,  \lambda \neq 1\} \\
&=& \sup \{|\lambda| : \lambda \in \sigma (T), \lambda \neq 1\}.
\end{eqnarray*}
This completes the proof.
\end{proof}

From this theorem we infer that how the spectral radius of $T-P$ relates to the spectral radius of the operator $T$.

\begin{remark}We point out that when $X=\br^n$ and $P$ is a rank one projection, then it is well-known \cite{Num} that $\beta^*$ is the best
possible rate of the convergence of $T^n$ to $P$. If $X=\ell_1$, then there is $\beta>\beta^*$ and $C>0$ such that $\|T^n-P\|\leq C\beta^n$ \cite{IL}.
Next
result refitments this fact in a general setting.
\end{remark}

\begin{corollary}
Let $T \in \Sigma (X)$ and $P$ be a projection. If $T$ is uniformly P-ergodic,  then $\|T^n - P\| = (\beta^* + \alpha_n)^n$ where $\alpha_n \to 0$ as $n\to \infty$.
\end{corollary}

\begin{proof}
Let $\alpha_n = \|T^n - P\|^{1/n} - \beta^*$ then by Theorem \ref{rateofspectrumtheorem}, we get the assertion.
\end{proof}

This result yields that if $T$ is uniformly $P$-ergodic, then the best rate of convergence is given by $r(T-P)$ and that the other rate is
$[\d_P(T^n)]^{1/n}$ which will never smaller that $r(T-P)$, since $r(T-P)=\inf [\d_P(T^n)]^{1/n}$.
We stress that if $\d_P(T)<1$, then this quantity is the best rate to use since it is much easier to calculate than the spectrum of $T$ or the spectral
radius of $T-P$. However, there are also some cases when $\d_P(T)=r(T-Q)$ which means that $\d_P(T)$ is the best possible.

\begin{theorem}
Let $T\in \Sigma (X)$ and $P$ be a Markov projection. Assueme that $T$ is uniformly $P$-ergodic. Then $\d_P (T) = r(T-P)$ if and only if $\d_P (T^n) = (\d_P(T))^n$ for all $n\in\mathbb N$.
\end{theorem}

\begin{proof}
Assume that $\d_P (T^n) = (\d_P(T))^n$ for all $n\in\mathbb N$. Then $(\d_P (T^n))^{1/n} = \d_P(T)$ for all $n\in\mathbb N$ and $\d_P(T) = \lim_{n\to\infty} (\d_P (T^n))^{1/n} = r(T-P)$.

Conversely, let us suppose that $r(T-P) = \d_P  (T)$. The uniform $P$-ergodicity implies $TP=PT=P$, therefore, $T^n - P = (T-P)^n $. Hence,
$$
r(T^n - P) = r ((T-P)^n) \leq \d_P(T^n) \leq (\d_P(T))^n = (r(T-P))^n.
$$
By the Spectral Mapping Theorem, we have $(r(T-P))^n = r(T^n-P)$, which yields
$$
\d_P (T^n)= (\d_P(T))^n, \  \forall n\in\mathbb N.
$$
This completes the proof.
\end{proof}

Let us provide an application of the obtained results.\\

Let $(E,\mathcal{F},\m)$ is an arbitrary probability space. By
$L^1(E,\m)$ we denote the usual $L^1$-space. Let $(X, X_+, \ck,
f)$ be an abstract state space. Consider $\tilde X:=L^1(E,\m; X)$
-- $L^1$-space of all $X$-valued measurable functions on
$(E,\mathcal{F},\m)$. The positive cone of this space is defined
usually, i.e. $\tilde X_+=L^1(E,\m; X_+)$. The generating
functional is defined as follows
$$
\tilde f(x)=\int f(x(t))d\m(t), \ \ x=x(t)\in L^1(E,\m: X).
$$

The base of $L^1(E,\m; X)$ is given by
$$
\tilde\ck=\{x\in L^1(E,\m; X_+): \ \tilde f(x)=1\}
$$
Then one can see that $(\tilde X,\tilde X_+,\tilde\ck,\tilde f)$
is an abstract state space.

Let $P(x,A)$ be a transition probability which defines a Markov
operator $S$ on $L^1(E,\m)$, whose dual $S^*$ acts on
$L^\infty(E,\m)$ as follows
$$
(S^*f)(x)=\int f(y)P(x,dy), \ \ f\in L^\infty.
$$
We assume that $S$ is uniformly $Q$-ergodic.

Now consider a Markov operator $T:X\to X$ which is uniformly $P$-ergodic for some Markov projection on $X$.
Using these two $S$ and $T$ operators, we define a new linear operator $\tilde T:\tilde X\to\tilde X$ whose dual acts on
$L^\infty(E,\m; X^*)$ as follows:
$$
\tilde T^*g(x)=\int P(x,dy)T^*g(y), \ \ g\in L^\infty(E,\m; X^*).
$$
We may look at $\tilde T$ by other way. One can see that the space $L^1(E,\m; X)$ can be treated as $L^1(E,\m)\otimes X$, and then the operator
$\tilde T$ is defined by $\tilde T=S\otimes T$. Now, using the standard argument (see \cite{M12}) we can establish that $\tilde T$ is uniformly
$Q\otimes P$-ergodic. Hence, by Theorem \ref{rT} we infer that $r(\tilde T-Q\otimes P)<1$. This means if $r(S-Q)<1$ and $r(T-P)<1$, then
one gets $r(\tilde T-Q\otimes P)<1$ which is a' priori not evident. Moreover, this gives the best rate for the convergence of
$\|\tilde T-Q\otimes P\|$. We point out that one has
$$
r(S\otimes T-Q\otimes P)\leq \max\{r(S-Q),r(T-P)\}.
$$
Indeed, let us denote $Z=S-Q, R=T-P$. Then due to \cite[Theorem 5.2]{MA}, one has
\begin{equation}\label{RZ}
ZQ=QZ=0, \ \ \ PR=RP=0.
\end{equation}
Therefore,
$$
S\otimes T-Q\otimes P=Q\otimes R+Z\otimes P+Z\otimes R.
$$
According to Theorem \ref{rT} and \eqref{RZ}, we have
\begin{eqnarray*}
r(S\otimes T-Q\otimes P)&=&\lim_{N\to\infty}\|(Q\otimes R+Z\otimes P+Z\otimes R)^N\|^{1/N}\\[2mm]
&=&\lim_{N\to\infty}\|Q\otimes R^N+Z^N\otimes P+Z^N\otimes R^N\|^{1/N}\\[2mm]
&\leq &\lim_{N\to\infty}\bigg(\|R^N\|+\|Z^N+\|Z^N\|\|R^N\|\bigg)^{1/N}\\[2mm]
&=&\max\big\{r(R),r(Z),r(Z)r(R)\big\}\\
&=&\max\{r(R),r(Z)\}
\end{eqnarray*}
which yields the assertion.

\section{Deoblin's condition for uniform $P$-ergodicity}

Let us introduce an abstract analogue of the
well-known Doeblin's Condition \cite{Num}. In this section, for the sake of convenience, we assume that $(X,X_+,\ck,f)$ is a
strong abstract state space.

As before, let $P$ be a Markov projection on $X$, and let $T\in \Sigma_P (X)$. Let $Q: X\to X$ be a Markov projection. We write $Q\leq P$, if $Q=QP=PQ$.
We say that $T$ satisfies \textit{condition $\frak{D}_P$}: if there exists a constant
$\t\in(0,1]$, an integer $n_0\in\bn$ and a Markov projection $Q$ with $Q\leq P$ such that for every $x\in \ck$  there exists $\f_{x}\in X_{+}$ with
$$
\sup\limits_{x}\|\f_{x}\|\leq \frac{\t}{4}
$$
such that
\begin{equation}\label{Dm}
T^{n_0}x+\f_{x}\geq\t Qx.
\end{equation}

The next result characterize the uniformly $P$-ergodic Markov operators in terms of the above condition $\frak{D}$.\\

\begin{theorem}\label{Dm1}  Let $(X,X_+,\ck,f)$ be a strong abstract state space, and let $P$ be a Markov projection on $X$.
Assume that $T\in \Sigma_P (X)$ and $TP=P$. Then the following conditions are equivalent:
\begin{enumerate}
\item[(i)]  $T$ satisfies condition $\frak{D}_P$;

 \item[(ii)] $T$ is uniformly $P$-ergodic.
\end{enumerate}
\end{theorem}

\begin{proof} (i) $\Rightarrow$ (ii).
By condition $\frak{D}_P$, there is a $\t\in (0,1]$, $n_0\in\bn$ and $Q\leq P$ such that for any two elements
$x,y\in\ck$ with $x-y\in N_P$, there exist
$\f_{x}\in\ck$, $\f_{y}\in X_{+}$ with
\begin{eqnarray}\label{Dm2}
\sup\limits_{x}\|\f_{x}\|\leq \frac{\t}{4}
\end{eqnarray}
such that
\begin{equation}\label{Dm3}
T^{n_0}x+\f_{x}\geq\t Qx, \ \ T^{n_0}y+\f_{y}\geq\t Qy.
\end{equation}

By putting $\f_{xy}=\f_{x}+\f_{x}$, from \eqref{Dm3} we obtain
\begin{equation}\label{Dm31}
T^{n_0}x+\f_{xy}\geq\t Qx, \ \ T^{n_0}y+\f_{xy}\geq\t Qy.
\end{equation}

By the Markovianity of $T$, and \eqref{Dm31}, \eqref{Dm2},  one gets
\begin{eqnarray}\label{Dm32}
\|T^{n_0}x+\f_{xy}-\t Qx\|&=&f(T^{n_0}x+\f_{xy}-\t Qx)\nonumber \\[2mm]
&=&1-\t+f(\f_{xy})\nonumber \\[2mm]
&\leq &1-\frac{\t}{2}.
\end{eqnarray}
By the same argument, one finds
\begin{eqnarray}\label{Dm33}
\|T^{n_0}y+\f_{xy}-\t Qy\|\leq 1-\frac{\t}{2}.
\end{eqnarray}

Due to $P(x-y)=0$ and $Q\leq P$, one gets $Qx=Qy$. Therefore, from \eqref{Dm32},\eqref{Dm33}, we obtain
\begin{eqnarray*}
\|T^{n_0}x-T^{n_0}y\|&=&\|T^{n_0}x+\f_{xy}-\t Qx-(T^{n_0}x+\f_{xy}-\t Qy)\|\\[2mm]
&\leq & \|T^{n_0}x+\f_{xy}-\t Qx\|+\|T^{n_0}x+\f_{xy}-\t Qy\|\\[2mm]
&\leq & 2\bigg(1-\frac{\t}{2}\bigg)
\end{eqnarray*}
this by \eqref{SS1} implies
$$
\d_P(T^{n_0})\leq 1-\frac{\t}{2}<1.
$$
Hence, by Theorem \ref{rT} we arrive at (ii).

(ii) $\Rightarrow$ (i).
 Assume that $T$ is uniformly $P$-mean ergodic.  Then
\begin{equation*}
\sup_{x\in\ck}\|T^{n}x-Px\|\to 0\ \ \ \textrm{as} \ \
n\to\infty.
\end{equation*}
Therefore, one can find $n_0\in\bn$ such that
\begin{equation}\label{N27}
\|T^{n_0}x-Px\|\leq\frac{1}{4}, \ \ \textrm{for all} \ \ x\in\ck.
\end{equation}

Then, for any $x\in\ck$, we decompose
\begin{eqnarray}\label{Dm4}
T^{n_0}x-Px=(T^{n_0}x-Px)_+-(T^{n_0}x-Px)_-
\end{eqnarray}
Denote
$$
\f_{x}=(T^{n_0}x-Px)_-.
$$
It is clear that $\f_{x}\in X_+$, and from \eqref{N27} one gets
$$
\sup_{x\in \ck}\|\f_{x}\| \leq\frac{1}{4}.
$$
Moreover, by \eqref{Dm4} we obtain
\begin{eqnarray*}
T^{n_0}x+\f_{x}&=&Px+T^{n_0}x-Px+\f_{x}\\[2mm]
&=&Px+(T^{n_0}(T)x-Px)_+\\
 &\geq& Px
\end{eqnarray*}
which means that $T$ satisfies the condition $\frak{D}_P$. This completes the proof.
\end{proof}

\begin{remark} We notice that the Deoblin's condition for $T$ has been investigated in
\cite{DP,M,M01,SG,SZ,Sz} when $P$ was taken rank one projection. We think that such type of result is even a new in the classical, i.e.  $X$ is taken as an $L^1$-space.
\end{remark}

Let us consider an other condition similar to $\frak{D}$. Namely, let $P$ and $T$ be as before. We say that $T$ satisfies \textit{condition $\frak{D}_P^*$}: if there exists a constant
$\l\in(0/2,1]$, an integer $n_0\in\bn$ and a Markov projection $Q$ with $Q\leq P$  such that for every $x\in \ck$ there exists $u_x\in X_+$ with $\|u_x\|\geq \l$
such that
\begin{equation}\label{DPm}
T^{n_0}x\geq u_x, \ \ Qx\geq u_x.
\end{equation}

\begin{theorem}\label{DPm1}  Let $(X,X_+,\ck,f)$ be a strong abstract state space, and let $P$ be a Markov projection on $X$.
Assume that $T\in \Sigma_P (X)$ and $TP=P$. If $T$  satisfies condition $\frak{D}_P^*$, then $T$ is uniformly $P$-ergodic.
\end{theorem}

\begin{proof} By condition $\frak{D}_P^*$, there is a $\l\in (1/2,1]$, $n_0\in\bn$  and for any two elements
$x,y\in\ck$ with $x-y\in N_P$, there exist
$u_x,u_yin X_+$, with $\|u_x|\geq \l$, $\|u_y\|\geq \l$
such that
\begin{eqnarray}\label{DPm2}
&& T^{n_0}x\geq u_x, \ \ Qx \geq u_x,\\[2mm]\label{DPm3}
&& T^{n_0}y\geq u_y, \ \ QPy \geq v_y.
\end{eqnarray}

Markovianity of $T$ with \eqref{DPm2} implies
\begin{eqnarray*}
\|T^{n_0}x-Qx\|&=&\|T^{n_0}x-u-(Qx-u)\|\\[2mm]
&\leq & \|T^{n_0}x-u\|+\|Qx-u\|\\[2mm]
&=&1-f(u)+1-f(u)\\[2mm]
&\leq &2(1-\l).
\end{eqnarray*}
By the same argument with \eqref{DPm3}, one finds
\begin{eqnarray*}
\|T^{n_0}x-Qx\|\leq 2(1-\l).
\end{eqnarray*}

Now using $Px=Py$ and $Q\leq P$, we obtain
\begin{eqnarray*}
\|T^{n_0}x-T^{n_0}y\|\leq \|T^{n_0}x-Qx\|+\|T^{n_0}y-Qy\|\leq 4(1-\l)
\end{eqnarray*}

Hence, by \eqref{SS1} we arrive at
$$
\d_P(T^{n_0})\leq 2(1-\l)<1.
$$
which by Theorem \ref{rT} yields the assertion.
\end{proof}

We notice that even the condition $\frak{D}_P$ is sufficient, but sometimes it could be practical rather than $\frak{D}$.

\appendix

\section{Examples of abstract State Spaces}

 Let us provide some examples of
abstract state spaces.

\begin{example}\label{eL1}
Let $X=L^1(E,\m)$ be the classical $L^1$-space. Then $X$ with the usual order and with
$$
f(x)=\int xd\m
$$
is an abstract state space.
\end{example}

\begin{example}\label{E1} Now, let us consider a more general example.
  Let $M$ be a von Neumann algebra. Let $M_{h,*}$ be the
Hermitian part of the predual space $M_*$ of $M$. As a base $\ck$
we define the set of normal states of $M$. Then
$(M_{h,*},M_{*,+},\ck,\id)$ is a strong abstract state spaces,
where $M_{*,+}$ is the set of all positive functionals taken from
$M_*$, and $\id$ is the unit in $M$. In particular, if
$M=L^\infty(E,\m)$, then $M_{*}=L^1(E,\m)$ is an abstract state
space.
\end{example}


\begin{example}\label{E13}
Let $X$ be a Banach space over $\br$. Consider a new Banach space
$\mathcal{X}=\br\oplus X$ with a norm
$\|(\a,x)\|=\max\{|\a|,\|x\|\}$. Define a cone
$\mathcal{X}_+=\{(\a,x)\ : \ \|x\|\leq \a, \ \a\in\br_+\}$ and a
positive functional $f(\a,x)=\a$. Then one can define a base
$\ck=\{(\a,x)\in\mathcal{X}:\ f(\a,x)=1\}$. Clearly, we have
$\ck=\{(1,x):\ \|x\|\leq 1\}$. Then
$(\mathcal{X},\mathcal{X}_+,\ck,f)$ is an abstract state space
\cite{Jam}. Moreover, $X$ can be isometrically embedded into
$\mathcal{X}$.  Using this construction one can study several
interesting examples of abstract state spaces.
\end{example}

\begin{example}\label{E14}  Let $A$ be the disc algebra, i.e. the sup-normed space
of complex-valued functions which are continuous on the closed
unit disc, and analytic on the open unit disc. Let $X =\{f\in A :\
f(1)\in\br \}$. Then $X$ is a real Banach space with the following
positive cone $X_+=\{f\in X: f(1)=\|f\|\}=\{f\in X: \
f(1)\geq\|f\|\}$. The space $X$ is an abstract state space, but
not strong one (see  \cite{Yo} for details).
\end{example}

\section{Examples of Markov operators}

 Let us consider several examples of Markov operators.

\begin{example}\label{E2}

Let $X=L^1(E,\m)$ be the classical $L^1$-space. Then any
transition probability $P(x,A)$ defines a Markov operator $T$ on
$X$, whose dual $T^*$ acts on $L^\infty(E,\m)$ as follows \cite{K}
$$
(T^*f)(x)=\int f(y)P(x,dy), \ \ f\in L^\infty.
$$
\end{example}

\begin{example}\label{E22}
Let $M$ be a von Neumann algebra, and consider
$(M_{h,*},M_{*,+},\ck,\id)$ as in Example \ref{E1}. Let $\Phi:
M\to M$ be a positive, unital ($\Phi(\id)=\id$) linear mapping.
Then the operator given by $(Tf)(x)=f(\Phi(x))$, where $f\in
M_{h,*}, x\in M$, is a Markov operator. Certain explicite examples of Markov operators can be found in \cite{CPR,M}.
\end{example}

\begin{example}\label{E23} Let $X=C[0,1]$ be the space of real-valued continuous
functions on $[0,1]$. Denote
$$
X_+=\big\{x\in X: \ \max_{0\leq t\leq 1}|x(t)-x(1)|\leq 2
x(1)\big\}.
$$
Then $X_+$ is a generating cone for $X$, and $f(x)=x(1)$ is a
strictly positive linear functional. Then $\ck=\{x\in X_+: \
f(x)=1\}$ is a base corresponding to $f$. One can check that the
base norm $\|x\|$ is equivalent to the usual one
$\|x\|_{\infty}=\max\limits_{0\leq t\leq 1}|x(t)|$. Due to
closedness of $X_+$ we conclude that $(X,X_+,\ck,f)$ is an
abstract state space. Let us define a mapping $T$ on $X$ as
follows:
$$(Tx)(t)=tx(t).
$$
It is clear that $T$ is a Markov operator on $X$.
\end{example}

\begin{example}\label{E24} Let $X$ be a Banach space over $\br$. Consider the
abstract state space $(\mathcal{X},\mathcal{X}_+,\tilde \ck,f)$
constructed in Example \ref{E13}.  Let $T:X\to X$ be a linear
bounded operator with $\|T\|\leq 1$. Then the operator
$\mathcal{T}: \mathcal{X}\to\mathcal{X}$ defined by
$\mathcal{T}(\a,x)=(\a,Tx)$ is a Markov operator. More concrete examples of such type of Markov operators have been studied in \cite{M0}.
\end{example}

\begin{example}\label{E25} Let $A$ be the disc algebra, and let $X$ be the
abstract state space as in Example \ref{E14}. A mapping $T$ given
by $Tf(z)=zf(z)$ is clearly a Markov operator on $X$.
\end{example}

\bibliographystyle{amsplain}

\end{document}